\documentclass[11pt]{amsart}
\usepackage[pass]{geometry}
\usepackage{verbatim}
\usepackage{ulem}

\usepackage{amsfonts,amsmath,amssymb,amscd,amsthm}
\usepackage{bbm}
\usepackage{mathtools}
\usepackage{graphicx,subcaption}
\usepackage{xcolor}
\usepackage{tikz}
\usepackage{enumitem}
\DeclarePairedDelimiter{\abs}{\lvert}{\rvert}

\definecolor{cvdred}{RGB}{220, 50, 32}
\usetikzlibrary{arrows.meta}
\usepackage{lineno}
\usepackage{multirow}
\usepackage{dsfont}
\usepackage{rotating}
\usepackage{float}
\usepackage{placeins}
\usepackage{pdfpages}
\usepackage{setspace}
\usepackage{comment}

\allowdisplaybreaks
\usepackage{appendix}
\usepackage{chngcntr}
\usepackage[colorlinks=true, allcolors=blue]{hyperref}
\usepackage{nameref}
\makeatletter
\newcommand\setcurrentname[1]{\def\@currentlabelname{#1}}
\makeatother

\newcommand{\N}{\mathbb{N}}

\makeatletter
\def\thm@space@setup{%
  \thm@preskip=4pt
  \thm@postskip=4pt
}
\makeatother
\theoremstyle{plain}

\newtheorem{theorem}{Theorem}

\newtheorem{corollary}{Corollary}
\newtheorem{conjecture}{Conjecture}
\newtheorem{example}{Example}

\makeatletter
\renewenvironment{proof}[1][\proofname]{%
  \par
  \pushQED{\qed}%
  \normalfont
  \topsep=4pt
  \trivlist
  \item[\hskip\labelsep\itshape #1\@addpunct{.}]%
  \ignorespaces
}{%
  \popQED
  \endtrivlist
  \@endpefalse
}

\newenvironment{evidence}
  {\begin{proof}[Evidence]}
  {\end{proof}}
\makeatother
\theoremstyle{definition}
\newtheorem{definition}{Definition}
\title{Computer-assisted results for \(n\)-sided non-transitive dice taking face values in \([n]\)}
\author{Terry Easlick}
\address{D\'epartement de math\'ematiques et de statistique, Universit\'e de Montr\'eal, Montr\'eal, Canada}
\thanks{Correspondence: \texttt{terry.easlick@umontreal.ca}.}
\date{\today}
\begin{document}
\begin{abstract}
We consider non-transitivity among dice with \(n\) sides restricted to taking face values in the set \(\{1,\ldots,n\}\). We give a computer-assisted proof that \(n=6\) is the smallest \(n\ge3\) which admits a non-transitive cycle of maximum length \(n\) with uniform and minimal winning advantage. For \(n=6\), we identify the two unique non-transitive cycles of length \(6\). We present computational evidence supporting a conjecture that \(n=6\) is the largest \(n\ge3\) which admits a non-transitive cycle of maximum length \(n\) with uniform and minimal winning advantage. 
\end{abstract}
\maketitle
\section{Introduction}
Non-transitive dice are fascinating mathematical objects that have attracted attention and have been studied by various researchers (cf. \cite{Gardner2001,AngelDavis2017,Grime2017,SchaeferSchweig2017,BuhlerGrahamHales2018,HurKim2021,KimEtAl2024}). In this paper, we are concerned specifically with non-transitive dice which have \(n\)-many sides and can only take values from the set \(\{1,\ldots,n\}\). Hereafter, all dice are assumed fair and independent with regard to outcomes of competing dice tosses. 
A standard \(n\)-sided die is such that each face takes a value from \(1\) to \(n\) for some natural number \(n \in \mathbb{N}\) with \(n\ge 3\), and each face value appears exactly once, and each face has equal probability of \(\frac{1}{n}\) to be a result in a die toss. We denote by \([n]\) the set \(\{1,\ldots,n\}\) with \(n\ge 3\). 
\begin{definition}
We will call a \textit{non-standard \(n\)-sided die} to be a die taking face values in \([n]\) such that face values may be repeated. Each face still possesses probability \(\frac{1}{n}\) to be the result of a toss.
\end{definition}
We will represent dice by \(n\)-tuples such that the coordinates are placed in non-decreasing order. 
\begin{example}
    The die in ordered tuple form \((1,2,3,4,5,6)\) is standard and will also occur in the set of non-standard dice; whereas, the die \((1,2,3,3,4,4)\) is non-standard. In particular, the set of standard \(n\)-sided dice is a singleton.
\end{example}

Given two non-standard \(n\)-sided dice, \(D_{\mathbf{x}} \coloneq (x_1,x_2,\ldots ,x_n)\) and \(D_{\mathbf{y}} \coloneq (y_1,y_2,\ldots ,y_n)\) for \(n\ge 3\), we consider the two cases \newline
1) If \(n\) is even, then there is some \(k\ge 2\) such that \(n=2k\). We take the following dice
\begin{align*}
    D_{\mathbf{x}} = (\underbrace{1,\ldots,1}_{k-1},2,\underbrace{3,\ldots,3}_{k})\quad\mathrm{and}\quad D_{\mathbf{y}} = (1,\underbrace{2,\ldots,2}_{2k-1}),
\end{align*}
then both are valid non-standard \(n\)-sided dice and comparing pairwise when face values of \(D_{\mathbf{x}}\) are strictly greater than the face values of \(D_{\mathbf{y}}\) yields
\begin{align*}
    k\cdot 2k + 1 = 2\left(\frac{n}{2}\right)^2 + 1  =\frac{n^2}{2}+1.
\end{align*}
2) If \(n\) is odd, then there is some \(k\ge 1\) such that \(n=2k+1\). We take the following dice
\begin{align*}
    D_{\mathbf{x}} = (\underbrace{1,\ldots,1}_{k},2,\underbrace{3,\ldots,3}_{k})\quad\mathrm{and}\quad D_{\mathbf{y}} = (\underbrace{1,\ldots,1}_{k+1},\underbrace{2,\ldots,2}_{k}),
\end{align*}
then both are valid non-standard \(n\)-sided dice and comparing pairwise when the face values of \(D_{\mathbf{x}}\) are strictly greater than \(D_{\mathbf{y}}\) yields
\begin{align*}
    k\cdot (2k+1) + (k+1) = n\frac{n-1}{2} + \frac{n-1}{2} + 1 = \frac{n^2+1}{2}. 
\end{align*}

Now, we can verify from both cases the resulting values are integers. Given such dice there are \(n^2\) equally likewise pairwise comparisons. Thus an algebraic evaluation yields that the smallest attainable winning probability strictly greater than \(\frac{1}{2}\) in favour of \(D_{\mathbf{x}}\) over \(D_{\mathbf{y}}\) is
\begin{align}\label{cases}
    \begin{cases}
        \frac{1}{2}+\frac{1}{n^2}, \quad n\ge3,n\bmod{2}=0, \\
        \frac{1}{2}+\frac{1}{2n^2}, \quad n\ge3,n\bmod{2}=1.
    \end{cases}
\end{align}

Let \(D_{\mathbf{x}}\) and \(D_{\mathbf{y}}\) be two non-standard \(n\)-sided dice. We say that \(D_{\mathbf{x}}\) \textit{beats} \(D_{\mathbf{y}}\) if the outcome of an independent toss of both dice results in the face value of \(D_{\mathbf{x}}\) being strictly greater than the face value of \(D_{\mathbf{y}}\).
\begin{definition}
A \textit{non-transitive cycle} of length \(k\ge 3\) is a sequence of distinct non-standard \(n\)-sided dice \(D_1,D_2,\ldots,D_k\) such that
\begin{align}\label{nontr}
    \mathbb{P}(D_i\text{ beats }D_{i+1})>\frac{1}{2}
\end{align}
for every \(i\in\{1,\ldots,k\}\), where \(D_{k+1}=D_1\). We call it a \textit{uniform} winning advantage when each probability in (\ref{nontr}) is identical and \textit{minimal} if it is the smallest attainable probability greater than \(\frac{1}{2}\) for the given \(n\).
\end{definition}

\begin{definition}
A directed graph \(G=(V,E)\) with no loops is called \textit{strongly connected} if for every pair of distinct vertices \(u,v\in V\), there is a directed path from \(u\) to \(v\) and from \(v\) to \(u\).

Given a directed graph \(G=(V,E)\) with no loops, a maximal subset \(S \subseteq V\) is called a \textit{strongly connected component} (SCC) which induces a directed subgraph such that for every pair of distinct vertices \(u, v\) in the directed subgraph, there is a directed path from \(u\) to \(v\) and from \(v\) to \(u\) both of which lie entirely in the subgraph induced by \(S\). The trivial SCCs are those for which \(S\) is a singleton.
\end{definition}
Strong connectivity defines an equivalence relation on
the vertex set \(V\) of a directed graph \(G\).
Thus, the SCCs are the equivalence classes of this relation
and partition \(V\). Namely, for two distinct SCCs \(S_1\) and \(S_2\) there may either be an out-edge or an in-edge but not in both directions, there may also be no edges between them. If an SCC has at least two vertices then it possesses a directed cycle; moreover, every directed cycle will be entirely within some SCC.

\section{Uniform and Minimal Non-Transitive cycle for \(n=6\) }
\begin{theorem}\label{thm1} \(n=6\) is the smallest \(n\ge3\) such that the set of non-standard \(n\)-sided dice admits a non-transitive cycle of maximum length \(n\) with uniform and minimal winning advantage.
\end{theorem}
\begin{proof}
We utilize (\ref{cases}) to consider what the minimal probability is depending on whether \(n\) is even or odd. The proof proceeds by constructing the necessary graphs and computer-assistance aids in enumerating all possibilities to exhaust the search space. 

\noindent\textbf{Construction of vertex set \(V\)}\newline
 For \(n\), a non-standard die is constructed by choosing \(n\) face values where each face value is from \([n]\). Since order does not matter, this is a combination with repetition counting problem. Thus, there are \(\binom{n+(n-1)}{n}\) many dice. Additionally, each such die has a unique non-decreasing ordered \(n\)-tuple representation. Hence, it is feasible to compute a list of all possible dice using a lexicographic ordering. 
 This list gives the vertex set for a directed graph where each vertex can be denoted by \(D_{\mathbf{x}}\) where \(\mathbf{x}\) is the non-decreasing ordered \(n\)-tuple. 
 
\noindent\textbf{Construction of edge set \(E\)}\newline
We utilize the following characterizations
for even \(n\)
    \begin{align}\label{favour}
        \mathbb{P}(D_x \,\mathrm{ beats }\, D_y) = \frac{1}{2}+\frac{1}{n^2} = \frac{n^2/2+1}{n^2} &\iff \sum_{i=1}^n\sum_{j=1}^n \mathds{1}_{\{x_i > y_j\}} = \frac{n^2}{2}+1,
    \end{align}
and for odd \(n\)
    \begin{align}\label{favour2}
        \mathbb{P}(D_x \,\mathrm{ beats }\, D_y) = \frac{1}{2}+\frac{1}{2n^2} = \frac{n^2+1}{2n^2} &\iff \sum_{i=1}^n\sum_{j=1}^n \mathds{1}_{\{x_i > y_j\}} = \frac{n^2+1}{2}.
    \end{align}
    
 Thus for each pair \((D_{\mathbf{x}},D_{\mathbf{y}}) \in V \times V\), we form a directed edge \(D_{\mathbf{x}} \to D_{\mathbf{y}} \) whenever the condition in (\ref{favour}) or (\ref{favour2})  is satisfied.

 Using the constructed vertex and and edge sets, we can construct the directed graphs. For a given \(n\) we will denote the graph by \(G_n\).

 \noindent\textbf{Computer-assistance}\newline
 Proceeding by exhaustive search, we use a script written in the Julia programming language which utilizes the readily available package \texttt{Graphs.jl} (cf. Appendix A), we construct the directed graph \(G=(V,E)\) on all possible non-standard \(n\)-sided dice which satisfy condition (\ref{favour}) or (\ref{favour2}) depending on the parity of \(n\). The script instantiates Tarjan's algorithm (cf. \cite{Tarjan1972}) to find all SCCs in \(G\). 
For verification purposes, the program outputs the following
\begin{enumerate}[topsep=3pt, itemsep=0pt]
    \item The total number of non-standard dice; thus, the number of vertices in \(G\).
    \item The total number of edges in \(G\) with verification each directed edge occurs only when the double sum in condition (\ref{favour}) or (\ref{favour2})  is satisfied.
    \item The total number of SCCs found by Tarjan's algorithm and the number of vertices and edges within each SCC found.
    \item Verification that the SCCs partition the constructed vertex set \(V\) and that each directed cycle lies entirely within an SCC.
    \item Each non-standard dice as the non-decreasing ordered \(n\)-tuple representation in each SCC which possesses at least \(n\)-many vertices.
    \item The complete list of all directed cycles within each SCC which has at least \(n\)-many vertices. 
\end{enumerate}

\noindent\textbf{Case for \(n=3\)}\newline
 For \(n=3\) there are \(\binom{3+2}{3} = 10\) total non-standard dice. The output of our program yields there are \(10\) distinct SCCs; namely, each SCC is trivial and thus none may contain a directed cycle, let alone one of length \(3\) (cf. Appendix B). In this case, the constructed graph \(G_3\) possesses \(10\) vertices and \(3\) directed edges.
 
\noindent\textbf{Case for \(n=4\)}\newline
 For \(n=4\) there are \(\binom{4+3}{4} = 35\) total non-standard dice. The output of our program yields there are \(35\) distinct SCCs; namely, each SCC is trivial and thus none may contain a directed cycle, let alone one of length \(4\) (cf. Appendix C). \(G_4\) contains \(35\) vertices and \(67\) directed edges.

\noindent\textbf{Case for \(n=5\)}\newline
   For \(n=5\) there are \(\binom{5+4}{5} = 126\) total non-standard dice. The output of our program yields there are \(126\) distinct SCCs; namely, each SCC is trivial and thus none may contain a directed cycle, let alone one of length \(5\) (cf. Appendix D). \(G_5\) contains \(126\) vertices and \(620\) directed edges.

\noindent\textbf{Case for \(n=6\)}\newline
 For \(n=6\), our program found \(\binom{11}{6} = 462\) non-standard \(6\)-sided dice. The constructed graph \(G_6\) has \(\abs{V}=462\) and \(\abs{E}=5491\). The number of SCCs is \(450\). Two SCCs each possess \(7\) vertices and \(10\) edges. The remaining \(448\) SCCs are all singleton vertices. Hereafter, we refer to these two non-trivial SCCs as SCC \(1\) and SCC \(2\) (cf. Appendix E).\newline

\begin{center}

\begin{tikzpicture}[
    >=Stealth,
    vertex/.style={
        draw,
        rounded corners,
        align=center,
        font=\scriptsize,
        inner sep=3pt
    },
    edge/.style={->, thick},
    cycle/.style={->, very thick, red}
]

\node[vertex] (v88)  at (0,4)
    { $v_{88}$\\$(1,1,2,5,5,5)$};

\node[vertex] (v217) at (4,2)
    {$v_{217}$\\$(1,3,4,4,4,4)$};

\node[vertex] (v277) at (4,-2)
    {$v_{277}$\\$(2,2,2,3,6,6)$};

\node[vertex] (v89) at (0,-4)
    {$v_{89}$\\$(1,1,2,5,5,6)$};

\node[vertex] (v207) at (-4,-2)
    {$v_{207}$\\$(1,3,3,4,4,4)$};

\node[vertex] (v291) at (-4,2)
    {$v_{291}$\\$(2,2,3,3,3,6)$};

\node[vertex] (v151) at (0,0)
    {$v_{151}$\\$(1,2,2,3,6,6)$};

\draw[edge] (v88)  to[bend right=12] (v207);
\draw[edge] (v207) to (v291);
\draw[edge] (v291) to (v88);

\draw[edge] (v88)  to (v217);
\draw[edge] (v217) to (v151);
\draw[edge] (v151) to (v88);

\draw[edge] (v217) to (v277);
\draw[edge] (v277) to (v89);
\draw[edge] (v89)  to (v207);

\draw[edge] (v89) to[bend left=18] (v217);

\draw[cycle] (v88)  to (v217);
\draw[cycle] (v217) to (v277);
\draw[cycle] (v277) to (v89);
\draw[cycle] (v89)  to (v207);
\draw[cycle] (v207) to (v291);
\draw[cycle] (v291) to (v88);

\end{tikzpicture}
\end{center}

Above is the diagram of SCC \(1\) for \(n=6\) which has \(7\) vertices and \(10\) edges. The directed cycle of length \(6\) is shown in red. The vertex number and its non-decreasing ordered \(6\)-tuple are given. \newline

\begin{center}
    
\begin{tikzpicture}[
    >=Stealth,
    vertex/.style={
        draw,
        rounded corners,
        align=center,
        font=\scriptsize,
        inner sep=3pt
    },
    edge/.style={->, thick},
    cycle/.style={->, very thick, red}
]

\node[vertex] (v118) at (0,4)
    {$v_{118}$\\$(1,1,4,5,5,5)$};

\node[vertex] (v385) at (4,2)
    {$v_{385}$\\$(3,3,3,3,4,6)$};

\node[vertex] (v286) at (4,-2)
    {$v_{286}$\\$(2,2,2,5,6,6)$};

\node[vertex] (v235) at (0,-4)
    {$v_{235}$\\$(1,4,4,4,5,5)$};

\node[vertex] (v391) at (-4,-2)
    {$v_{391}$\\$(3,3,3,4,4,6)$};

\node[vertex] (v160) at (-4,2)
    {$v_{160}$\\$(1,2,2,5,6,6)$};

\node[vertex] (v119) at (0,0)
    {$v_{119}$\\$(1,1,4,5,5,6)$};

\draw[edge] (v118) to (v385);
\draw[edge] (v385) to[bend right=12] (v160);
\draw[edge] (v160) to (v118);

\draw[edge] (v385) to (v286);
\draw[edge] (v286) to (v235);
\draw[edge] (v235) to (v391);
\draw[edge] (v391) to (v160);

\draw[edge] (v286) to (v119);
\draw[edge] (v119) to (v385);

\draw[edge] (v391) to[bend left=18] (v286);

\draw[cycle] (v118) to (v385);
\draw[cycle] (v385) to (v286);
\draw[cycle] (v286) to (v235);
\draw[cycle] (v235) to (v391);
\draw[cycle] (v391) to (v160);
\draw[cycle] (v160) to (v118);

\end{tikzpicture}
\end{center}

Similarly to SCC \(1\), the diagram of SCC \(2\) for \(n=6\) has \(7\) vertices and \(10\) edges. The directed cycle of length \(6\) is shown in red. The vertex number and its non-decreasing ordered \(6\)-tuple are given.

Both SCC \(1\) and SCC \(2\) contain non-transitive cycles of length \(6\). Since these are the only nontrivial SCCs and each contains exactly \(7\) vertices, it is impossible to have a non-transitive cycle greater than \(7\). Moreover, given the complete enumeration, it is possible to verify by hand both these SCCs do not contain a non-transitive cycle of length \(7\). Therefore, we have established that \(n=6\) satisfies the theorem statement. The program exhausts the search space using the prescribed construction of the necessary directed graphs thus the proof is complete (cf. Appendices A,B,C,D,E).
\end{proof}

\begin{definition}
    We call a mapping \(f\) from a directed graph \(G = (V,E)\) to itself an \textit{antiautomorphism} if maps the set \(f: V \to V\) is a bijection,  and for every \(u,v \in V\) we have 
    \begin{align*}
        u\to v\in E \iff f(v)\to f(u)\in E.
    \end{align*}
    We call it an \textit{involutive antiautomorphism} if it is self-inverse.
\end{definition}

Now we define
\begin{align}\label{invo}
    f(D_{\mathbf{x}})\coloneq f(x_1,\ldots,x_6) = (7-x_6,\ldots,7-x_1).
\end{align}

\begin{corollary}
    The map \(f\) is an involutive antiautomorphism of \(G_6\). In particular, it interchanges SCC 1 and SCC 2.
\end{corollary}
\begin{proof}
Firstly,
\begin{align*}
    f(f(D_{\mathbf{x}})) = f((7-x_6,\ldots,7-x_1)) = (x_1,\ldots,x_6) = D_{\mathbf{x}}.
\end{align*}
Thus it is involutive.

Secondly, we have for \(G_6 = (V,E)\)
\begin{align*}
f(D_{\mathbf{y}}) \to f(D_{\mathbf{x}}) & \iff 
\sum_{i=1}^{6}\sum_{j=1}^{6}
\mathds{1}_{\{7-y_{7-i}>7-x_{7-j}\}} = 19 \\
& \iff \sum_{i=1}^{6}\sum_{j=1}^{6}
\mathds{1}_{\{x_{7-j}>y_{7-i}\}} = 19 \\
&\iff \sum_{i=1}^{6}\sum_{j=1}^{6}
\mathds{1}_{\{x_i>y_j\}} = 19 \\
&\iff D_{\mathbf{x}} \to D_{\mathbf{y}}
\end{align*}
Thus it reverses edges within \(G_6\).

Lastly, an antiautomorphism must send SCCs to SCCs and here we see \(f\) sends SCC 1 to SCC 2, and vice versa.
\begin{align*}
(1,1,2,5,5,5)
&\mapsto
(7-5,\,7-5,\,7-5,\,7-2,\,7-1,\,7-1)
=(2,2,2,5,6,6),\\
(1,3,4,4,4,4)
&\mapsto
(7-4,\,7-4,\,7-4,\,7-4,\,7-3,\,7-1)
=(3,3,3,3,4,6),\\
(2,2,2,3,6,6)
&\mapsto
(7-6,\,7-6,\,7-3,\,7-2,\,7-2,\,7-2)
=(1,1,4,5,5,5),\\
(1,1,2,5,5,6)
&\mapsto
(7-6,\,7-5,\,7-5,\,7-2,\,7-1,\,7-1)
=(1,2,2,5,6,6),\\
(1,3,3,4,4,4)
&\mapsto
(7-4,\,7-4,\,7-4,\,7-3,\,7-3,\,7-1)
=(3,3,3,4,4,6),\\
(2,2,3,3,3,6)
&\mapsto
(7-6,\,7-3,\,7-3,\,7-3,\,7-2,\,7-2)
=(1,4,4,4,5,5).
\end{align*}
where the LHS side is the directed \(6\)-cycle from SCC \(1\) read from top to bottom and the RHS side is the directed \(6\)-cycle from SCC \(2\) read from bottom to top.
\end{proof}

\section{Conjecture for \(n=6\) and Evidence for \(n>6\)}
\begin{conjecture}\label{con1}
\(n=6\) is the largest \(n\ge3\) such that the set of non-standard \(n\)-sided dice admits a non-transitive cycle of maximum length \(n\) with uniform and minimal winning advantage.
\end{conjecture}
\begin{evidence}
We use a similar approach to that done for the proof of Theorem~\ref{thm1}. We construct the directed graphs for \(n=7,8,9,10\) and find the SCCs of the resulting graphs.

\noindent\textbf{Computer-assistance}\newline
As before we use a program written in the Julia programming language and the package \texttt{Graphs.jl} (cf. Appendix F). Again by way of exhaustive search, we use the program to construct the directed graph \(G_n\) for the given \(n\). The construction of edges follow from using the following criteria
\begin{align}\label{favour3}
       \sum_{i=1}^n\sum_{j=1}^n \mathds{1}_{\{x_i > y_j\}} = \sum_{i=1}^n\abs{\{j : x_j > i\}}\abs{\{j : y_j = i\}} 
\end{align} This provided a speed-up in obtaining the results (cf. the Technology Disclosure below). Again,
Tarjan's algorithm (cf. \cite{Tarjan1972}) is utilized to find the SCCs. Once the SCCs are found, the program searches each SCC with at least \(n+1\) vertices using a depth-first search (DFS) until it finds a directed cycle which has length \(n+1\) or greater. Given the non-standard dice are represented in non-decreasing ordered \(n\)-tuples placed in a lexicographical ordering, the program chooses the first vertex label under this ordering that occurs in the SCC possessing at least \(n+1\) vertices and performs the DFS, if no such non-transitive cycle is found it moves on to the second vertex in this ordering, and so forth until such a non-transitive cycle is either found or not found. The program stops as soon as a non-transitive cycle of length \(n+1\) or greater is found; otherwise it exhausts all cases. If no SCC is found to contain at least \(n+1\) vertices then there is no directed cycle of length at least \(n+1\). The output of the program is as follows
\begin{enumerate}[topsep=3pt, itemsep=0pt]
     \item The total number of non-standard dice; thus, the number of vertices in \(G\).
    \item The total number of edges in \(G\), with verification each directed edge occurs only when the sum on the RHS in condition (\ref{favour3}) is equal to the value given in (\ref{favour}) or (\ref{favour2}), dependent upon the parity of \(n\).
    \item The total number of SCCs found by Tarjan's algorithm and the number of vertices and edges within each SCC found.
    \item Verification that the SCCs partition the constructed vertex set \(V\).
    \item Return the first found directed cycle of length at least n+1.
 \end{enumerate}

 Following are cases for \(n=7,8,9,10\). In each case we find suitable non-transitive cycles which are at least \(n+1\) in length.
 
\noindent\textbf{Case for \(n=7\)}\newline
The graph \(G_7\) has \(1,716\) vertices and \(89,153\) directed edges. It possesses \(636\) SCCs of which one has \(1081\) vertices and \(53,389\) edges. All other SCCs are trivial (cf. Appendix G). Here is a non-transitive cycle of length \(14\)

\begin{center}
\begin{tikzpicture}[
    >=Stealth,
    die/.style={
        draw,
        rounded corners,
        font=\scriptsize,
        inner sep=3pt,
        minimum width=4.1cm
    },
    edge/.style={->, thick},
    x=4.5cm,
    y=1.5cm
]

% Row 1
\node[die] (d1) at (0,0) {$(1,1,1,3,3,4,7)$};
\node[die] (d2) at (1,0) {$(2,2,2,2,2,2,6)$};
\node[die] (d3) at (2,0) {$(1,1,1,3,5,5,5)$};

% Row 2
\node[die] (d4) at (2,-1) {$(1,1,2,2,4,4,4)$};
\node[die] (d5) at (1,-1) {$(1,1,3,3,3,3,3)$};
\node[die] (d6) at (0,-1) {$(1,2,2,2,2,5,6)$};

% Row 3
\node[die] (d7) at (0,-2) {$(1,1,1,4,4,4,5)$};
\node[die] (d8) at (1,-2) {$(1,1,3,3,3,3,4)$};
\node[die] (d9) at (2,-2) {$(1,2,2,2,2,5,7)$};

% Row 4
\node[die] (d10) at (2,-3) {$(1,1,1,4,4,4,6)$};
\node[die] (d11) at (1,-3) {$(1,1,3,3,3,3,5)$};
\node[die] (d12) at (0,-3) {$(1,2,2,2,2,6,7)$};

% Row 5
\node[die] (d13) at (0,-4) {$(1,1,1,3,5,5,6)$};
\node[die] (d14) at (1,-4) {$(2,2,2,2,3,4,4)$};

% Directed cycle
\draw[edge] (d1) -- (d2);
\draw[edge] (d2) -- (d3);
\draw[edge] (d3) -- (d4);
\draw[edge] (d4) -- (d5);
\draw[edge] (d5) -- (d6);
\draw[edge] (d6) -- (d7);
\draw[edge] (d7) -- (d8);
\draw[edge] (d8) -- (d9);
\draw[edge] (d9) -- (d10);
\draw[edge] (d10) -- (d11);
\draw[edge] (d11) -- (d12);
\draw[edge] (d12) -- (d13);
\draw[edge] (d13) -- (d14);

% Closing edge
\draw[edge]
    (d14.south)
    -- ++(0,-0.6)
    -| ([xshift=-0.75cm]d1.west)
    -- (d1.west);

\end{tikzpicture}
\end{center}

\noindent\textbf{Case for \(n=8\)}\newline
For \(n=8\), the number of vertices in \(G_8\) is \(6,435\) and number of edges are \(956,307\). Additionally, there was only one non-trivial SCC with \(4,321\) vertices and \(674,108\) edges (cf. Appendix H).  Using our program, we obtain the following directed cycle of length \(29\).\newline

\begin{center}
\begin{tikzpicture}[
    >=Stealth,
    die/.style={
        draw,
        rounded corners,
        font=\scriptsize,
        inner sep=3pt,
        minimum width=4.1cm
    },
    edge/.style={->, thick},
    x=4.5cm,
    y=1.5cm
]

% Row 1
\node[die] (d1) at (0,0) {$(1,1,1,2,4,6,6,6)$};
\node[die] (d2) at (1,0) {$(1,1,3,3,3,3,3,5)$};
\node[die] (d3) at (2,0) {$(2,2,2,2,2,4,4,4)$};

% Row 2
\node[die] (d4) at (2,-1) {$(1,1,1,3,3,3,7,8)$};
\node[die] (d5) at (1,-1) {$(2,2,2,2,2,2,6,7)$};
\node[die] (d6) at (0,-1) {$(1,1,1,2,5,5,5,6)$};

% Row 3
\node[die] (d7) at (0,-2) {$(1,2,3,3,3,4,4,4)$};
\node[die] (d8) at (1,-2) {$(2,2,2,2,2,3,6,6)$};
\node[die] (d9) at (2,-2) {$(1,1,1,2,5,5,5,7)$};

% Row 4
\node[die] (d10) at (2,-3) {$(1,2,3,3,4,4,4,4)$};
\node[die] (d11) at (1,-3) {$(1,2,2,2,3,3,6,6)$};
\node[die] (d12) at (0,-3) {$(1,1,1,2,5,5,5,5)$};

% Row 5
\node[die] (d13) at (0,-4) {$(1,2,3,4,4,4,4,4)$};
\node[die] (d14) at (1,-4) {$(1,2,3,3,3,3,4,6)$};
\node[die] (d15) at (2,-4) {$(2,2,2,2,2,5,5,5)$};

% Row 6
\node[die] (d16) at (2,-5) {$(1,1,1,3,4,4,8,8)$};
\node[die] (d17) at (1,-5) {$(2,2,2,2,2,3,6,7)$};
\node[die] (d18) at (0,-5) {$(1,1,1,2,5,5,5,8)$};

% Row 7
\node[die] (d19) at (0,-6) {$(1,2,4,4,4,4,4,4)$};
\node[die] (d20) at (1,-6) {$(1,1,1,3,3,6,7,7)$};
\node[die] (d21) at (2,-6) {$(2,2,2,2,2,5,5,6)$};

% Row 8
\node[die] (d22) at (2,-7) {$(1,1,1,4,4,4,6,7)$};
\node[die] (d23) at (1,-7) {$(1,1,3,3,3,3,5,6)$};
\node[die] (d24) at (0,-7) {$(2,2,2,2,2,4,5,8)$};

% Row 9
\node[die] (d25) at (0,-8) {$(1,1,1,3,3,7,7,7)$};
\node[die] (d26) at (1,-8) {$(2,2,2,2,2,2,6,8)$};
\node[die] (d27) at (2,-8) {$(1,1,1,3,5,5,5,6)$};

% Row 10
\node[die] (d28) at (2,-9) {$(1,3,3,3,3,4,4,4)$};
\node[die] (d29) at (1,-9) {$(2,2,2,3,3,3,3,7)$};

% Directed cycle
\draw[edge] (d1) -- (d2);
\draw[edge] (d2) -- (d3);
\draw[edge] (d3) -- (d4);
\draw[edge] (d4) -- (d5);
\draw[edge] (d5) -- (d6);
\draw[edge] (d6) -- (d7);
\draw[edge] (d7) -- (d8);
\draw[edge] (d8) -- (d9);
\draw[edge] (d9) -- (d10);
\draw[edge] (d10) -- (d11);
\draw[edge] (d11) -- (d12);
\draw[edge] (d12) -- (d13);
\draw[edge] (d13) -- (d14);
\draw[edge] (d14) -- (d15);
\draw[edge] (d15) -- (d16);
\draw[edge] (d16) -- (d17);
\draw[edge] (d17) -- (d18);
\draw[edge] (d18) -- (d19);
\draw[edge] (d19) -- (d20);
\draw[edge] (d20) -- (d21);
\draw[edge] (d21) -- (d22);
\draw[edge] (d22) -- (d23);
\draw[edge] (d23) -- (d24);
\draw[edge] (d24) -- (d25);
\draw[edge] (d25) -- (d26);
\draw[edge] (d26) -- (d27);
\draw[edge] (d27) -- (d28);
\draw[edge] (d28) -- (d29);

% Closing edge
\draw[edge]
    (d29.west)
    -- ++(-1.15,0)
    |- (d1.west);

\end{tikzpicture}
\end{center}

\noindent\textbf{Case for \(n=9\)}\newline
The graph \(G_9\) has \(24,310\) vertices and \(12,528,131\) directed edges. It possesses \(3,860\) SCCs of which one has \(20,451\) vertices and \(10,933,011\) edges (cf Appendix I). The remaining SCCs are singletons. Here is a non-transitive cycle of length \(10\)

\begin{center}
\begin{tikzpicture}[
    >=Stealth,
    die/.style={
        draw,
        rounded corners,
        font=\scriptsize,
        inner sep=3pt,
        minimum width=4.1cm
    },
    edge/.style={->, thick},
    x=4.5cm,
    y=1.5cm
]

% Row 1
\node[die] (d1) at (0,0) {$(1,1,1,1,3,3,3,3,7)$};
\node[die] (d2) at (1,0) {$(2,2,2,2,2,2,2,2,6)$};
\node[die] (d3) at (2,0) {$(1,1,1,1,3,5,5,5,5)$};

% Row 2
\node[die] (d4) at (2,-1) {$(1,1,2,2,2,4,4,4,4)$};
\node[die] (d5) at (1,-1) {$(1,1,1,2,3,3,3,3,8)$};
\node[die] (d6) at (0,-1) {$(2,2,2,2,2,2,2,2,7)$};

% Row 3
\node[die] (d7) at (0,-2) {$(1,1,1,1,3,3,6,6,6)$};
\node[die] (d8) at (1,-2) {$(2,2,2,2,2,2,2,4,5)$};
\node[die] (d9) at (2,-2) {$(1,1,1,1,3,3,4,5,9)$};

% Row 4
\node[die] (d10) at (2,-3) {$(2,2,2,2,2,2,2,2,8)$};

% Directed cycle
\draw[edge] (d1) -- (d2);
\draw[edge] (d2) -- (d3);
\draw[edge] (d3) -- (d4);
\draw[edge] (d4) -- (d5);
\draw[edge] (d5) -- (d6);
\draw[edge] (d6) -- (d7);
\draw[edge] (d7) -- (d8);
\draw[edge] (d8) -- (d9);
\draw[edge] (d9) -- (d10);

% Closing edge
\draw[edge]
    (d10.east)
    -- ++(0.75cm,0)
    |- ([yshift=0.7cm]d1.north)
    -- (d1.north);

\end{tikzpicture}
\end{center}

\noindent\textbf{Case for \(n=10\)}\newline
\(G_{10}\) is such that \(\abs{V} = 92,378 \) and \(\abs{E} = 157,244,770\). There was only one non-trivial SCC  with \(81,791\) vertices and \(147,620,052\) edges (cf. Appendix J). Here is a directed cycle of length \(110\) found in that SCC. 

\begin{center}
\resizebox{\textwidth}{!}{%
\begin{tikzpicture}[
    >=Stealth,
    die/.style={
        draw,
        rounded corners,
        font=\tiny,
        inner sep=2pt,
        minimum width=3.7cm
    },
    edge/.style={->,thick},
    x=4.0cm,
    y=1.0cm
]

\foreach \lab [count=\i] in {
{1,1,1,1,2,4,4,4,4,4},
{1,2,2,3,3,3,3,3,3,3},
{1,2,2,2,2,2,2,4,7,7},
{1,1,1,1,3,3,3,6,6,6},
{1,2,2,2,2,2,2,5,5,5},
{1,1,1,1,2,4,4,4,4,5},
{1,2,3,3,3,3,3,3,3,3},
{1,1,1,2,2,2,6,7,7,7},
{1,1,1,1,2,4,5,5,5,6},
{1,1,1,3,3,3,3,3,4,4},
{2,2,2,2,2,2,2,3,6,6},
{1,1,1,1,2,5,5,5,5,6},
{1,2,2,2,2,2,4,4,4,4},
{1,1,1,3,3,3,3,3,3,5},
{2,2,2,2,2,2,2,4,4,7},
{1,1,1,1,3,3,3,6,6,7},
{2,2,2,2,2,2,2,4,5,5},
{1,1,1,1,3,3,3,4,7,8},
{2,2,2,2,2,2,2,2,6,7},
{1,1,1,1,3,4,4,5,5,6},
{1,2,2,3,3,3,3,3,3,4},
{1,2,2,2,2,2,2,4,8,8},
{1,1,1,1,3,3,3,6,7,7},
{2,2,2,2,2,2,2,5,5,5},
{1,1,1,2,2,4,4,4,4,4},
{1,1,1,3,3,3,3,3,3,6},
{2,2,2,2,2,2,2,4,4,8},
{1,1,1,1,3,3,3,4,7,9},
{2,2,2,2,2,2,2,2,6,8},
{1,1,1,1,2,5,5,5,5,7},
{1,2,2,2,2,3,4,4,4,4},
{1,1,1,3,3,3,3,3,3,7},
{1,2,2,2,2,2,2,5,6,7},
{1,1,1,1,3,4,4,5,5,5},
{2,2,2,3,3,3,3,3,3,4},
{1,1,1,2,2,2,6,7,7,8},
{1,1,1,1,3,4,5,5,5,6},
{1,1,2,3,3,3,3,3,4,4},
{2,2,2,2,2,2,2,3,6,7},
{1,1,1,1,2,5,5,5,5,8},
{1,2,2,2,2,4,4,4,4,4},
{1,1,1,1,3,3,3,4,7,10},
{2,2,2,2,2,2,2,2,6,9},
{1,1,1,1,3,4,5,5,5,7},
{1,2,2,3,3,3,3,3,4,4},
{1,2,2,2,2,2,2,4,8,9},
{1,1,1,1,3,3,3,7,7,7},
{1,2,2,2,2,2,2,5,6,6},
{1,1,1,1,4,4,4,5,5,5},
{1,1,1,3,3,3,3,3,4,5},
{2,2,2,2,2,2,2,3,6,8},
{1,1,1,1,2,5,5,5,5,9},
{1,2,2,2,3,3,4,4,4,4},
{1,1,1,3,3,3,3,3,3,8},
{1,2,2,2,2,2,2,5,6,8},
{1,1,1,1,4,4,4,4,5,6},
{1,1,2,3,3,3,3,3,4,5},
{1,2,2,2,2,2,2,4,9,9},
{1,1,1,1,3,3,3,7,7,8},
{1,2,2,2,2,2,2,6,6,6},
{1,1,1,1,2,5,5,5,5,10},
{1,2,2,2,3,4,4,4,4,4},
{1,1,1,1,3,3,3,4,8,9},
{2,2,2,2,2,2,2,2,7,8},
{1,1,1,1,2,4,6,6,6,7},
{1,1,1,3,3,3,3,3,5,5},
{1,2,2,2,2,2,2,4,9,10},
{1,1,1,1,3,3,3,7,8,8},
{2,2,2,2,2,2,2,2,6,10},
{1,1,1,1,3,4,5,5,5,8},
{1,2,3,3,3,3,3,3,3,4},
{1,1,1,2,2,2,6,7,7,9},
{1,1,1,1,3,5,5,5,5,6},
{1,1,1,2,2,4,4,4,4,5},
{1,1,1,3,3,3,3,3,3,9},
{1,2,2,2,2,2,2,5,6,9},
{1,1,1,1,4,4,4,4,5,7},
{1,1,1,3,3,3,3,3,4,6},
{2,2,2,2,2,2,2,3,6,9},
{1,1,1,1,2,5,5,5,6,6},
{1,2,2,2,4,4,4,4,4,4},
{1,1,1,2,3,3,3,4,8,10},
{2,2,2,2,2,2,2,2,7,9},
{1,1,1,1,2,4,6,6,6,8},
{1,1,2,3,3,3,3,3,3,5},
{1,2,2,2,2,2,2,4,10,10},
{1,1,1,1,3,3,3,7,7,9},
{2,2,2,2,2,2,2,4,5,6},
{1,1,1,1,3,3,3,4,9,10},
{2,2,2,2,2,2,2,2,8,9},
{1,1,1,1,2,4,6,6,7,8},
{1,1,3,3,3,3,3,3,3,5},
{1,2,2,2,2,2,4,4,4,5},
{1,1,1,3,3,3,3,3,3,10},
{1,2,2,2,2,2,2,5,6,10},
{1,1,1,1,4,4,4,4,5,8},
{1,1,1,3,3,3,3,3,4,7},
{1,2,2,2,2,2,2,6,6,7},
{1,1,1,1,3,3,5,5,5,10},
{2,2,2,2,2,2,2,4,4,9},
{1,1,1,1,3,3,3,7,7,10},
{2,2,2,2,2,2,2,4,6,6},
{1,1,1,1,3,4,5,5,5,9},
{2,2,2,3,3,3,3,3,4,4},
{1,1,1,2,2,2,6,7,7,10},
{1,1,1,1,4,4,4,5,5,6},
{1,1,1,3,3,3,3,3,4,8},
{1,2,2,2,2,2,2,5,7,8},
{1,1,1,1,3,4,4,4,6,7},
{2,2,2,2,3,3,3,3,3,5}
}{
    \pgfmathtruncatemacro{\row}{floor((\i-1)/6)}
    \pgfmathtruncatemacro{\rawcol}{mod(\i-1,6)}

    \ifodd\row
        \pgfmathtruncatemacro{\col}{5-\rawcol}
    \else
        \pgfmathtruncatemacro{\col}{\rawcol}
    \fi

    \node[die] (d\i) at (\col,-\row) {$\bigl(\lab\bigr)$};

    \xdef\N{\i}
}

% Consecutive directed edges
\pgfmathtruncatemacro{\NmOne}{\N-1}

\foreach \i in {1,...,\NmOne}{
    \pgfmathtruncatemacro{\j}{\i+1}
    \draw[edge] (d\i) -- (d\j);
}

% Closing directed edge
\draw[edge]
    (d\N.south)
    -- ++(0,-0.6)
    -| ([xshift=-0.5cm]d1.west)
    -- (d1.west);

\end{tikzpicture}%
}
\end{center}

Continuing onto cases \(n=11\) and \(n=12\) was considerably more time consuming. The graph \(G_{11}\) has \(352,716\) vertices and the graph \(G_{12}\) has \(1,352,078\).  Since the search space is growing relatively quickly, we surmise both these graphs contain non-transitive cycles in some SCC of length at least \(12\) and \(13\), respectively. 

This concludes the evidence presented for the conjecture at this time (cf. Appendices F,G,H,I,J).

\end{evidence}

\section{Conclusion}
We have given a computer-assisted proof that \(n=6\) is the smallest value for which the restricted class of non-standard \(n\)-sided dice admits a non-transitive directed cycle of maximum length \(n\) with uniform and minimal winning advantage. The proof identifies the two unique directed \(6\)-cycles which satisfy the conditions of Theorem~\ref{thm1}.

For \(n=7,8,9\) and \(10\), the corresponding directed graphs contain cycles of length strictly greater than \(n\), providing computational evidence for Conjecture~\ref{con1}. The case for \(n=11\), \(12\) and beyond becomes computationally more difficult, and determining whether the conjecture holds in general remains open.

This work was inspired by many games of dice the author played with his mom which led to a fascination with (non-transitive) dice and interesting examples in probability and combinatorics.

\section*{Technology Disclosure}
The author utilized Google in the search for literature on non-transitive dice. The author utilized GPT-5.6 for aid in writing the Julia scripts as a means to save time and quick prototyping. The author prompted the GPT-5.6 LLM with explicit instructions for the writing of the scripts in Julia. Prompts included the statement ``no advice, no tangents, no unnecessary commentary'' to ensure the LLM completed only the task without deviation from the author-originated ideas. For evidence related to Conjecture~\ref{con1}, the author permitted the LLM to aid in speeding up the computer-assisted portion due to the long execution time for \(n = 10\), as a result the LLM augmented the script with the condition on the RHS of (\ref{favour3}). The same LLM was utilized for creating TikZ diagrams as a means to avoid the frustration of dealing with TikZ. The mathematical ideas and their development are entirely the author's own.  The author assumes full responsibility for the mathematical content, correctness of the arguments, interpretation of the computational results, and any errors that may remain. 

\appendix
\section{Code for Proof of Theorem~\ref{thm1}}
Be sure to edit \texttt{target} depending on the parity of \(n\).
\begin{verbatim}
    using Graphs

# ============================================================
# 1. Parameter
# ============================================================

n = 6

@assert n > 0  "n must be a positive integer"

target = n^2 ÷ 2 + 1 # (n^2+1) ÷ 2 if n is odd
outfile = "dice_cycles_n$(n).txt"

# ============================================================
# 2. Enumerate all nondecreasing n-tuples lexicographically
# ============================================================

function generate_dice(n)
    dice = Tuple[]
    x = zeros(Int, n)

    function build(pos, lower)
        if pos > n
            push!(dice, Tuple(x))
            return
        end

        for value in lower:n
            x[pos] = value
            build(pos + 1, value)
        end
    end

    build(1, 1)
    return dice
end

dice = generate_dice(n)
N = length(dice)

# Number of multisets of size n chosen from {1,...,n}
expected_vertices = binomial(2n - 1, n)

# ============================================================
# 3. Construct directed graph
#
#    X -> Y iff
#
#        sum_{i=1}^n sum_{j=1}^n 1{x_i > y_j}
#            = n^2/2 + 1
# ============================================================

beatscore(X, Y) = sum(x > y for x in X for y in Y)

g = SimpleDiGraph(N)
qualifying_edges = 0

for i in 1:N
    for j in 1:N
        i == j && continue

        if beatscore(dice[i], dice[j]) == target
            add_edge!(g, i, j)
            qualifying_edges += 1
        end
    end
end

# Verify every graph edge satisfies exactly the required condition.
edge_rule_ok =
    ne(g) == qualifying_edges &&
    all(beatscore(dice[src(e)], dice[dst(e)]) == target for e in edges(g))

# ============================================================
# 4. Tarjan strongly connected components
# ============================================================

components = strongly_connected_components_tarjan(g)

# Sort largest first; break ties by smallest vertex number.
sort!(components, by = c -> (-length(c), minimum(c)))

num_sccs = length(components)

# Number of internal edges in every SCC.
scc_edges = Int[]

for C in components
    sg, _ = induced_subgraph(g, C)
    push!(scc_edges, ne(sg))
end

# ============================================================
# 5. Verify that the SCCs partition the entire vertex set
# ============================================================

all_scc_vertices = vcat(components...)

partition_ok =
    length(all_scc_vertices) == N &&
    length(unique(all_scc_vertices)) == N &&
    sort(all_scc_vertices) == collect(1:N)

# ============================================================
# 6. SCCs having at least n vertices
# ============================================================

selected_scc_indices = findall(
    C -> length(C) >= n,
    components
)

# ============================================================
# 7. Exhaustively enumerate all directed simple cycles
#    in every SCC having at least n vertices
# ============================================================

function cycles_in_scc(g, component)
    component = sort(component)

    sg, vertex_map = induced_subgraph(g, component)

    local_cycles = simplecycles(sg)

    # Convert subgraph vertex numbers back to original graph vertices.
    return [
        [vertex_map[v] for v in cycle]
        for cycle in local_cycles
    ]
end

selected_scc_cycles = [
    cycles_in_scc(g, components[k])
    for k in selected_scc_indices
]

cycle_containment_ok = all(
    all(v in Set(components[k]) for v in cycle)
    for (r, k) in enumerate(selected_scc_indices)
    for cycle in selected_scc_cycles[r]
)

# ============================================================
# 8. Write complete report
# ============================================================

open(outfile, "w") do io

    println(io, "DIRECTED DICE GRAPH — EXHAUSTIVE SCC/CYCLE REPORT")
    println(io, "="^65)

    println(io, "\nn = ", n)
    println(io, "Required double-sum value = n^2/2 + 1 = ", target)
    println(io, "Edge condition:")
    println(io,
        "X -> Y iff sum_{i=1}^n sum_{j=1}^n 1{x_i > y_j} = ",
        target
    )

    # --------------------------------------------------------
    # Vertices
    # --------------------------------------------------------

    println(io, "\n", "="^65)
    println(io, "1. VERTICES")
    println(io, "="^65)

    println(io, "Total number of dice/vertices: ", N)
    println(io, "Expected binomial(2n-1,n):     ", expected_vertices)
    println(io,
        "Vertex-count verification:      ",
        N == expected_vertices ? "PASS" : "FAIL"
    )

    # --------------------------------------------------------
    # Edges
    # --------------------------------------------------------

    println(io, "\n", "="^65)
    println(io, "2. EDGES")
    println(io, "="^65)

    println(io, "Total number of directed edges: ", ne(g))
    println(io, "Required double-sum value:       ", target)
    println(io,
        "Edge-condition verification:     ",
        edge_rule_ok ? "PASS" : "FAIL"
    )

    # --------------------------------------------------------
    # SCCs
    # --------------------------------------------------------

    println(io, "\n", "="^65)
    println(io, "3. STRONGLY CONNECTED COMPONENTS")
    println(io, "="^65)

    println(io, "Number of SCCs: ", num_sccs)

    println(io, "\nSCC sizes:")
    for k in eachindex(components)
        println(
            io,
            "SCC ", k,
            ": vertices = ", length(components[k]),
            ", internal edges = ", scc_edges[k]
        )
    end

    # --------------------------------------------------------
    # Partition verification
    # --------------------------------------------------------

    println(io, "\n", "="^65)
    println(io, "4. SCC PARTITION VERIFICATION")
    println(io, "="^65)

    println(io,
        "SCCs partition the full vertex set: ",
        partition_ok ? "PASS" : "FAIL"
    )

    println(io,
        "Hence every directed cycle is entirely contained in one SCC."
    )

    println(io,
        "Cycle-containment check for enumerated cycles: ",
        cycle_containment_ok ? "PASS" : "FAIL"
    )

    # --------------------------------------------------------
    # SCCs with at least n vertices
    # --------------------------------------------------------

    println(io, "\n", "="^65)
    println(io, "5. SCCs WITH AT LEAST n VERTICES")
    println(io, "="^65)

    println(io, "Threshold n: ", n)
    println(io,
        "Number of SCCs having at least n vertices: ",
        length(selected_scc_indices)
    )

    for (r, k) in enumerate(selected_scc_indices)
        C = sort(components[k])

        println(io, "\nSelected SCC ", r)
        println(io, "Original SCC number: ", k)
        println(io, "Number of vertices: ", length(C))
        println(io, "Number of internal edges: ", scc_edges[k])

        println(io, "\nDice/tuples in this SCC:")

        for v in C
            println(io, "  vertex ", v, ": ", dice[v])
        end
    end

    # --------------------------------------------------------
    # Complete cycle lists
    # --------------------------------------------------------

    println(io, "\n", "="^65)
    println(io, "6. ALL DIRECTED SIMPLE CYCLES IN SCCs WITH AT LEAST n VERTICES")
    println(io, "="^65)

    for (r, k) in enumerate(selected_scc_indices)
        cycles = selected_scc_cycles[r]

        println(io, "\nSelected SCC ", r)
        println(io, "Original SCC number: ", k)
        println(io, "Number of vertices: ", length(components[k]))
        println(io, "Number of directed simple cycles: ", length(cycles))

        if isempty(cycles)
            println(io, "No directed cycles.")
            continue
        end

        for (j, cycle) in enumerate(cycles)
            println(io, "\nCycle ", j, " (length ", length(cycle), "):")

            println(
                io,
                "  vertices: ",
                join(cycle, " -> "),
                " -> ",
                first(cycle)
            )

            println(
                io,
                "  dice:     ",
                join((string(dice[v]) for v in cycle), " -> "),
                " -> ",
                dice[first(cycle)]
            )
        end
    end
end

println("Results written to: ", outfile)
\end{verbatim}

\section{Julia Output for \(n=3\)}
\begin{verbatim}
    DIRECTED DICE GRAPH — EXHAUSTIVE SCC/CYCLE REPORT
=================================================================

n = 3
Required double-sum value = (n^2+1)/2 = 5
Edge condition:
X -> Y iff sum_{i=1}^n sum_{j=1}^n 1{x_i > y_j} = 5

=================================================================
1. VERTICES
=================================================================
Total number of dice/vertices: 10
Expected binomial(2n-1,n):     10
Vertex-count verification:      PASS

=================================================================
2. EDGES
=================================================================
Total number of directed edges: 3
Required double-sum value:       5
Edge-condition verification:     PASS

=================================================================
3. STRONGLY CONNECTED COMPONENTS
=================================================================
Number of SCCs: 10

SCC sizes:
SCC 1: vertices = 1, internal edges = 0
SCC 2: vertices = 1, internal edges = 0
SCC 3: vertices = 1, internal edges = 0
SCC 4: vertices = 1, internal edges = 0
SCC 5: vertices = 1, internal edges = 0
SCC 6: vertices = 1, internal edges = 0
SCC 7: vertices = 1, internal edges = 0
SCC 8: vertices = 1, internal edges = 0
SCC 9: vertices = 1, internal edges = 0
SCC 10: vertices = 1, internal edges = 0

=================================================================
4. SCC PARTITION VERIFICATION
=================================================================
SCCs partition the full vertex set: PASS
Hence every directed cycle is entirely contained in one SCC.
Cycle-containment check for enumerated cycles: PASS

=================================================================
5. SCCs WITH AT LEAST n VERTICES
=================================================================
Threshold n: 3
Number of SCCs having at least n vertices: 0

=================================================================
6. ALL DIRECTED SIMPLE CYCLES IN SCCs WITH AT LEAST n VERTICES
=================================================================

\end{verbatim}

\section{Julia Output for \(n=4\)}
\begin{verbatim}
    DIRECTED DICE GRAPH — EXHAUSTIVE SCC/CYCLE REPORT
=================================================================

n = 4
Required double-sum value = n^2/2 + 1 = 9
Edge condition:
X -> Y iff sum_{i=1}^n sum_{j=1}^n 1{x_i > y_j} = 9

=================================================================
1. VERTICES
=================================================================
Total number of dice/vertices: 35
Expected binomial(2n-1,n):     35
Vertex-count verification:      PASS

=================================================================
2. EDGES
=================================================================
Total number of directed edges: 67
Required double-sum value:       9
Edge-condition verification:     PASS

=================================================================
3. STRONGLY CONNECTED COMPONENTS
=================================================================
Number of SCCs: 35

SCC sizes:
35 trivial SCCs (one vertex, zero internal edges each).

=================================================================
4. SCC PARTITION VERIFICATION
=================================================================
SCCs partition the full vertex set: PASS
Hence every directed cycle is entirely contained in one SCC.
Cycle-containment check for enumerated cycles: PASS

=================================================================
5. SCCs WITH AT LEAST n VERTICES
=================================================================
Threshold n: 4
Number of SCCs having at least n vertices: 0

=================================================================
6. ALL DIRECTED SIMPLE CYCLES IN SCCs WITH AT LEAST n VERTICES
=================================================================


\end{verbatim}

\section{Julia Output for \(n=5\)}
\begin{verbatim}
    DIRECTED DICE GRAPH — EXHAUSTIVE SCC/CYCLE REPORT
=================================================================

n = 5
Required double-sum value = (n^2+1)/2 = 13
Edge condition:
X -> Y iff sum_{i=1}^n sum_{j=1}^n 1{x_i > y_j} = 13

=================================================================
1. VERTICES
=================================================================
Total number of dice/vertices: 126
Expected binomial(2n-1,n):     126
Vertex-count verification:      PASS

=================================================================
2. EDGES
=================================================================
Total number of directed edges: 620
Required double-sum value:       13
Edge-condition verification:     PASS

=================================================================
3. STRONGLY CONNECTED COMPONENTS
=================================================================
Number of SCCs: 126

SCC sizes:  (All 126 SCCs are singletons)

=================================================================
4. SCC PARTITION VERIFICATION
=================================================================
SCCs partition the full vertex set: PASS
Hence every directed cycle is entirely contained in one SCC.
Cycle-containment check for enumerated cycles: PASS

=================================================================
5. SCCs WITH AT LEAST n VERTICES
=================================================================
Threshold n: 5
Number of SCCs having at least n vertices: 0

=================================================================
6. ALL DIRECTED SIMPLE CYCLES IN SCCs WITH AT LEAST n VERTICES
=================================================================

\end{verbatim}

\section{Julia Output for \(n=6\)}
\begin{verbatim}
    DIRECTED DICE GRAPH — EXHAUSTIVE SCC/CYCLE REPORT
=================================================================

n = 6
Required double-sum value = n^2/2 + 1 = 19
Edge condition:
X -> Y iff sum_{i=1}^n sum_{j=1}^n 1{x_i > y_j} = 19

=================================================================
1. VERTICES
=================================================================
Total number of dice/vertices: 462
Expected binomial(2n-1,n):     462
Vertex-count verification:      PASS

=================================================================
2. EDGES
=================================================================
Total number of directed edges: 5491
Required double-sum value:       19
Edge-condition verification:     PASS

=================================================================
3. STRONGLY CONNECTED COMPONENTS
=================================================================
Number of SCCs: 450

SCC sizes:
SCC 1: vertices = 7, internal edges = 10
SCC 2: vertices = 7, internal edges = 10
448 trivial SCCs (one vertex, zero internal edges each).

=================================================================
4. SCC PARTITION VERIFICATION
=================================================================
SCCs partition the full vertex set: PASS
Hence every directed cycle is entirely contained in one SCC.
Cycle-containment check for enumerated cycles: PASS

=================================================================
5. SCCs WITH AT LEAST n VERTICES
=================================================================
Threshold n: 6
Number of SCCs having at least n vertices: 2

Selected SCC 1
Original SCC number: 1
Number of vertices: 7
Number of internal edges: 10

Dice/tuples in this SCC:
  vertex 88: (1, 1, 2, 5, 5, 5)
  vertex 89: (1, 1, 2, 5, 5, 6)
  vertex 151: (1, 2, 2, 3, 6, 6)
  vertex 207: (1, 3, 3, 4, 4, 4)
  vertex 217: (1, 3, 4, 4, 4, 4)
  vertex 277: (2, 2, 2, 3, 6, 6)
  vertex 291: (2, 2, 3, 3, 3, 6)

Selected SCC 2
Original SCC number: 2
Number of vertices: 7
Number of internal edges: 10

Dice/tuples in this SCC:
  vertex 118: (1, 1, 4, 5, 5, 5)
  vertex 119: (1, 1, 4, 5, 5, 6)
  vertex 160: (1, 2, 2, 5, 6, 6)
  vertex 235: (1, 4, 4, 4, 5, 5)
  vertex 286: (2, 2, 2, 5, 6, 6)
  vertex 385: (3, 3, 3, 3, 4, 6)
  vertex 391: (3, 3, 3, 4, 4, 6)

=================================================================
6. ALL DIRECTED SIMPLE CYCLES IN SCCs WITH AT LEAST n VERTICES
=================================================================

Selected SCC 1
Original SCC number: 1
Number of vertices: 7
Number of directed simple cycles: 4

Cycle 1 (length 3):
  vertices: 88 -> 207 -> 291 -> 88
  dice:     (1, 1, 2, 5, 5, 5) -> (1, 3, 3, 4, 4, 4) 
  -> (2, 2, 3, 3, 3, 6) -> (1, 1, 2, 5, 5, 5)

Cycle 2 (length 3):
  vertices: 88 -> 217 -> 151 -> 88
  dice:     (1, 1, 2, 5, 5, 5) -> (1, 3, 4, 4, 4, 4) 
  -> (1, 2, 2, 3, 6, 6) -> (1, 1, 2, 5, 5, 5)

Cycle 3 (length 6):
  vertices: 88 -> 217 -> 277 -> 89 -> 207 -> 291 -> 88
  dice:     (1, 1, 2, 5, 5, 5) -> (1, 3, 4, 4, 4, 4) 
  -> (2, 2, 2, 3, 6, 6) -> (1, 1, 2, 5, 5, 6) 
  -> (1, 3, 3, 4, 4, 4) -> (2, 2, 3, 3, 3, 6) -> (1, 1, 2, 5, 5, 5)

Cycle 4 (length 3):
  vertices: 217 -> 277 -> 89 -> 217
  dice:     (1, 3, 4, 4, 4, 4) -> (2, 2, 2, 3, 6, 6) 
  -> (1, 1, 2, 5, 5, 6) -> (1, 3, 4, 4, 4, 4)

Selected SCC 2
Original SCC number: 2
Number of vertices: 7
Number of directed simple cycles: 4

Cycle 1 (length 3):
  vertices: 118 -> 385 -> 160 -> 118
  dice:     (1, 1, 4, 5, 5, 5) -> (3, 3, 3, 3, 4, 6) 
  -> (1, 2, 2, 5, 6, 6) -> (1, 1, 4, 5, 5, 5)

Cycle 2 (length 6):
  vertices: 118 -> 385 -> 286 -> 235 -> 391 -> 160 -> 118
  dice:     (1, 1, 4, 5, 5, 5) -> (3, 3, 3, 3, 4, 6) 
  -> (2, 2, 2, 5, 6, 6) -> (1, 4, 4, 4, 5, 5) 
  -> (3, 3, 3, 4, 4, 6) -> (1, 2, 2, 5, 6, 6) -> (1, 1, 4, 5, 5, 5)

Cycle 3 (length 3):
  vertices: 385 -> 286 -> 119 -> 385
  dice:     (3, 3, 3, 3, 4, 6) -> (2, 2, 2, 5, 6, 6) 
  -> (1, 1, 4, 5, 5, 6) -> (3, 3, 3, 3, 4, 6)

Cycle 4 (length 3):
  vertices: 286 -> 235 -> 391 -> 286
  dice:     (2, 2, 2, 5, 6, 6) -> (1, 4, 4, 4, 5, 5) 
  -> (3, 3, 3, 4, 4, 6) -> (2, 2, 2, 5, 6, 6)

\end{verbatim}
\section{Code for evidence of Conjecture~\ref{con1}}
Be sure to edit \texttt{target} depending on the parity of \(n\).
\begin{verbatim}
    using Graphs

# ============================================================
# Exact directed-dice graph and long-cycle search
# ============================================================

function main(n)

    @assert n > 0  "n must be a positive integer"

    target = n^2 ÷ 2 + 1 # (n^2 + 1) ÷ 2 if n is odd
    min_cycle_length = n + 1
    outfile = "dice_cycle_search_n$(n).txt"

    # ========================================================
    # 1. Enumerate all nondecreasing n-tuples
    #    in lexicographic order
    # ========================================================

    function generate_dice(n)
        dice = Tuple[]
        x = zeros(Int, n)

        function build(pos, lower)
            if pos > n
                push!(dice, Tuple(x))
                return
            end

            for value in lower:n
                x[pos] = value
                build(pos + 1, value)
            end
        end

        build(1, 1)
        return dice
    end

    dice = generate_dice(n)
    N = length(dice)

    expected_vertices = binomial(2n - 1, n)

    # Map every die to its lexicographic vertex label.
    index_of = Dict{Tuple,Int}()

    for (i, D) in enumerate(dice)
        index_of[D] = i
    end

    # ========================================================
    # 2. Original double-sum edge score
    # ========================================================

    function beatscore(X, Y)
        j = 1
        score = 0

        @inbounds for x in X
            while j <= n && Y[j] < x
                j += 1
            end
            score += j - 1
        end

        return score
    end

    # ========================================================
    # 3. Generate exactly the outgoing neighbors of X
    #
    # If b_s is the number of faces of Y equal to s, and
    #
    #     w_s = #{i : x_i > s},
    #
    # then
    #
    #     score(X,Y) = sum_s w_s b_s.
    #
    # Hence X -> Y exactly when
    #
    #     sum_s b_s = n,
    #     sum_s w_s b_s = target,
    #     b_s >= 0.
    #
    # The recursion below enumerates exactly these solutions.
    # ========================================================

    function exact_outneighbors(X)

        # w[s] = number of faces of X strictly greater than s
        w = zeros(Int, n)

        p = 1
        @inbounds for s in 1:n
            while p <= n && X[p] <= s
                p += 1
            end
            w[s] = n - p + 1
        end

        counts = zeros(Int, n)
        y = zeros(Int, n)
        neighbors = Int[]

        # Convert multiplicities to the corresponding sorted die.
        function emit_neighbor()
            k = 1

            @inbounds for s in 1:n
                for _ in 1:counts[s]
                    y[k] = s
                    k += 1
                end
            end

            Y = Tuple(y)
            push!(neighbors, index_of[Y])
        end

        function search_counts(s, faces_left, score_left)

            # Last face value is n, and w[n] = 0.
            if s == n
                if score_left == 0
                    counts[n] = faces_left
                    emit_neighbor()
                end
                return
            end

            ws = w[s]
            next_max = w[s + 1]

            for c in 0:faces_left

                remaining_score = score_left - c * ws
                remaining_faces = faces_left - c

                remaining_score < 0 && break

                # Remaining faces can contribute at most
                # remaining_faces * next_max.
                if remaining_score <= remaining_faces * next_max
                    counts[s] = c
                    search_counts(
                        s + 1,
                        remaining_faces,
                        remaining_score
                    )
                end
            end

            counts[s] = 0
        end

        if target <= n * w[1]
            search_counts(1, n, target)
        end

        return neighbors
    end

    # ========================================================
    # 4. Construct the exact directed graph
    # ========================================================

    g = SimpleDiGraph(N)

    generated_edge_count = 0

    for i in 1:N

        neighbors = exact_outneighbors(dice[i])

        for j in neighbors
            i == j && continue

            generated_edge_count += 1
            add_edge!(g, i, j)
        end
    end

    # Verify no duplicate generated edges were suppressed and
    # every generated edge satisfies the original double sum.
    edge_rule_ok =
        ne(g) == generated_edge_count &&
        all(
            beatscore(dice[src(e)], dice[dst(e)]) == target
            for e in edges(g)
        )

    # ========================================================
    # 5. Tarjan strongly connected components
    # ========================================================

    components = strongly_connected_components_tarjan(g)

    sort!(
        components,
        by = C -> (-length(C), minimum(C))
    )

    num_sccs = length(components)

    scc_of = zeros(Int, N)

    for (k, C) in enumerate(components)
        for v in C
            scc_of[v] = k
        end
    end

    # Number of internal edges in each SCC.
    scc_edges = zeros(Int, num_sccs)

    for e in edges(g)
        a = scc_of[src(e)]
        b = scc_of[dst(e)]

        if a == b
            scc_edges[a] += 1
        end
    end

    # ========================================================
    # 6. Verify that SCCs partition the vertex set
    # ========================================================

    all_scc_vertices = vcat(components...)

    partition_ok =
        length(all_scc_vertices) == N &&
        length(unique(all_scc_vertices)) == N &&
        sort(all_scc_vertices) == collect(1:N)

    # SCCs large enough possibly to contain a cycle of
    # length at least n+1.
    eligible_sccs = findall(
        C -> length(C) >= min_cycle_length,
        components
    )

    # ========================================================
    # 7. Exhaustive DFS for first simple directed cycle
    #    of length at least n+1
    #
    # Every cycle is contained in a single SCC.
    #
    # For each cycle, its smallest vertex is its canonical
    # starting point. The condition next > start prevents
    # rediscovery from another cyclic rotation.
    # ========================================================

    function find_long_cycle()

        for scc_index in eligible_sccs

            component = sort(components[scc_index])

            allowed = falses(N)
            allowed[component] .= true

            visited = falses(N)
            path = Int[]

            function dfs(start, current)

                for next in outneighbors(g, current)

                    allowed[next] || continue

                    if next == start

                        if length(path) >= min_cycle_length
                            return copy(path)
                        end

                    elseif !visited[next] && next > start

                        visited[next] = true
                        push!(path, next)

                        cycle = dfs(start, next)

                        if cycle !== nothing
                            return cycle
                        end

                        pop!(path)
                        visited[next] = false
                    end
                end

                return nothing
            end

            for start in component

                fill!(visited, false)
                empty!(path)

                visited[start] = true
                push!(path, start)

                cycle = dfs(start, start)

                if cycle !== nothing
                    return cycle, scc_index
                end
            end
        end

        return nothing, nothing
    end

    cycle, cycle_scc = find_long_cycle()

    # ========================================================
    # 8. Independent verification of returned cycle
    # ========================================================

    if cycle === nothing

        cycle_length_ok = false
        cycle_simple_ok = false
        cycle_edges_ok = false
        cycle_scores_ok = false
        cycle_scc_ok = false
        cycle_scores = Int[]

    else

        L = length(cycle)

        cycle_scores = [
            beatscore(
                dice[cycle[i]],
                dice[cycle[mod1(i + 1, L)]]
            )
            for i in 1:L
        ]

        cycle_length_ok = L >= min_cycle_length

        cycle_simple_ok =
            length(unique(cycle)) == L

        cycle_edges_ok = all(
            has_edge(
                g,
                cycle[i],
                cycle[mod1(i + 1, L)]
            )
            for i in 1:L
        )

        cycle_scores_ok =
            all(score == target for score in cycle_scores)

        cycle_scc_ok =
            all(scc_of[v] == cycle_scc for v in cycle)
    end

    # ========================================================
    # 9. Write rigorous report
    # ========================================================

    open(outfile, "w") do io

        println(io, "DIRECTED DICE GRAPH — EXACT LONG-CYCLE SEARCH")
        println(io, "="^72)
        #change Required edge score to (n^2+1)/2 when n is odd
        println(io, "\nn = ", n)
        println(
            io,
            "Required edge score n^2/2 + 1 = ",
            target
        )
        println(
            io,
            "Required cycle length >= n+1 = ",
            min_cycle_length
        )

        # ----------------------------------------------------
        # 1. Vertices
        # ----------------------------------------------------

        println(io, "\n", "="^72)
        println(io, "1. VERTICES")
        println(io, "="^72)

        println(
            io,
            "Total number of dice/vertices: ",
            N
        )

        println(
            io,
            "Expected binomial(2n-1,n):     ",
            expected_vertices
        )

        println(
            io,
            "Vertex-count verification:      ",
            N == expected_vertices ? "PASS" : "FAIL"
        )

        # ----------------------------------------------------
        # 2. Edges
        # ----------------------------------------------------

        println(io, "\n", "="^72)
        println(io, "2. EDGES")
        println(io, "="^72)

        println(
            io,
            "Total number of directed edges: ",
            ne(g)
        )

        println(
            io,
            "Required double-sum value:       ",
            target
        )

        println(
            io,
            "Exact neighbor generation uses:"
        )

        println(
            io,
            "  sum_s b_s = n"
        )

        println(
            io,
            "  sum_s w_s b_s = n^2/2 + 1"
        )

        println(
            io,
            "  b_s >= 0"
        )

        println(
            io,
            "where w_s = #{i : x_i > s}."
        )

        println(
            io,
            "\nGenerated-edge count agrees with graph edge count: ",
            ne(g) == generated_edge_count ? "PASS" : "FAIL"
        )

        println(
            io,
            "Every directed edge satisfies the original double-sum condition: ",
            edge_rule_ok ? "PASS" : "FAIL"
        )

        # ----------------------------------------------------
        # 3. SCCs
        # ----------------------------------------------------

        println(io, "\n", "="^72)
        println(io, "3. STRONGLY CONNECTED COMPONENTS")
        println(io, "="^72)

        println(
            io,
            "Total number of SCCs found by Tarjan's algorithm: ",
            num_sccs
        )

        println(io, "\nSCC sizes:")

        for k in eachindex(components)

            println(
                io,
                "SCC ", k,
                ": vertices = ", length(components[k]),
                ", internal edges = ", scc_edges[k]
            )
        end

        # ----------------------------------------------------
        # 4. Partition
        # ----------------------------------------------------

        println(io, "\n", "="^72)
        println(io, "4. SCC PARTITION VERIFICATION")
        println(io, "="^72)

        println(
            io,
            "SCCs partition the full vertex set: ",
            partition_ok ? "PASS" : "FAIL"
        )

        println(
            io,
            "Every directed cycle is therefore contained entirely within one SCC."
        )

        println(
            io,
            "Number of SCCs having at least n+1 vertices: ",
            length(eligible_sccs)
        )

        # ----------------------------------------------------
        # 5. Long-cycle search
        # ----------------------------------------------------

        println(io, "\n", "="^72)
        println(
            io,
            "5. SEARCH FOR A DIRECTED CYCLE OF LENGTH AT LEAST n+1"
        )
        println(io, "="^72)

        if isempty(eligible_sccs)

            println(
                io,
                "No SCC has at least ",
                min_cycle_length,
                " vertices."
            )

            println(
                io,
                "Therefore no simple directed cycle of length at least ",
                min_cycle_length,
                " exists."
            )

        elseif cycle === nothing

            println(
                io,
                "All eligible SCCs were exhaustively searched by DFS."
            )

            println(
                io,
                "No simple directed cycle of length at least ",
                min_cycle_length,
                " was found."
            )

        else

            println(io, "Cycle found: PASS")

            println(
                io,
                "SCC containing cycle: ",
                cycle_scc
            )

            println(
                io,
                "Size of containing SCC: ",
                length(components[cycle_scc])
            )

            println(
                io,
                "Cycle length: ",
                length(cycle)
            )

            println(
                io,
                "\nCycle length >= n+1: ",
                cycle_length_ok ? "PASS" : "FAIL"
            )

            println(
                io,
                "Cycle has no repeated vertices: ",
                cycle_simple_ok ? "PASS" : "FAIL"
            )

            println(
                io,
                "Every consecutive pair is a directed graph edge: ",
                cycle_edges_ok ? "PASS" : "FAIL"
            )

            println(
                io,
                "Every cycle edge has double-sum value ",
                target,
                ": ",
                cycle_scores_ok ? "PASS" : "FAIL"
            )

            println(
                io,
                "All cycle vertices lie in one SCC: ",
                cycle_scc_ok ? "PASS" : "FAIL"
            )

            println(io, "\nVertices:")

            println(
                io,
                join(cycle, " -> "),
                " -> ",
                first(cycle)
            )

            println(io, "\nDice/tuples:")

            for v in cycle
                println(
                    io,
                    "  vertex ",
                    v,
                    ": ",
                    dice[v]
                )
            end

            println(
                io,
                "  vertex ",
                first(cycle),
                ": ",
                dice[first(cycle)]
            )

            println(io, "\nVerification of each directed cycle edge:")

            L = length(cycle)

            for i in 1:L

                u = cycle[i]
                v = cycle[mod1(i + 1, L)]
                score = cycle_scores[i]

                println(io)
                println(
                    io,
                    "  ",
                    dice[u],
                    " -> ",
                    dice[v]
                )

                println(
                    io,
                    "    double-sum value = ",
                    score,
                    "  ",
                    score == target ? "PASS" : "FAIL"
                )
            end
        end
    end

    println("Results written to: ", outfile)
end

# ============================================================
# Run
# ============================================================

main(10) 
\end{verbatim}

\section{Julia Output for \(n=7\)}
\begin{verbatim}
    ========================================================================

n = 7
Required edge score (n^2+1)/2  = 25
Required cycle length >= n+1 = 8

========================================================================
1. VERTICES
========================================================================
Total number of dice/vertices: 1716
Expected binomial(2n-1,n):     1716
Vertex-count verification:      PASS

========================================================================
2. EDGES
========================================================================
Total number of directed edges: 89153
Required double-sum value:       25
Exact neighbor generation uses:
  sum_s b_s = n
  sum_s w_s b_s = (n^2+1)/2
  b_s >= 0
where w_s = #{i : x_i > s}.

Generated-edge count agrees with graph edge count: PASS
Every directed edge satisfies the original double-sum condition: PASS

========================================================================
3. STRONGLY CONNECTED COMPONENTS
========================================================================
Total number of SCCs found by Tarjan's algorithm: 636

SCC sizes:
SCC 1: vertices = 1081, internal edges = 53389
SCC 2 through SCC 636 are all singleton thus trivial (edited to save room)

========================================================================
4. SCC PARTITION VERIFICATION
========================================================================
SCCs partition the full vertex set: PASS
Every directed cycle is therefore contained entirely within one SCC.
Number of SCCs having at least n+1 vertices: 1

========================================================================
5. SEARCH FOR A DIRECTED CYCLE OF LENGTH AT LEAST n+1
========================================================================
Cycle found: PASS
SCC containing cycle: 1
Size of containing SCC: 1081
Cycle length: 14

Cycle length >= n+1: PASS
Cycle has no repeated vertices: PASS
Every consecutive pair is a directed graph edge: PASS
Every cycle edge has double-sum value 25: PASS
All cycle vertices lie in one SCC: PASS

Dice/tuples:
  vertex 149: (1, 1, 1, 3, 3, 4, 7)
  vertex 929: (2, 2, 2, 2, 2, 2, 6)
  vertex 166: (1, 1, 1, 3, 5, 5, 5)
  vertex 247: (1, 1, 2, 2, 4, 4, 4)
  vertex 337: (1, 1, 3, 3, 3, 3, 3)
  vertex 479: (1, 2, 2, 2, 2, 5, 6)
  vertex 177: (1, 1, 1, 4, 4, 4, 5)
  vertex 338: (1, 1, 3, 3, 3, 3, 4)
  vertex 480: (1, 2, 2, 2, 2, 5, 7)
  vertex 178: (1, 1, 1, 4, 4, 4, 6)
  vertex 339: (1, 1, 3, 3, 3, 3, 5)
  vertex 482: (1, 2, 2, 2, 2, 6, 7)
  vertex 167: (1, 1, 1, 3, 5, 5, 6)
  vertex 951: (2, 2, 2, 2, 3, 4, 4)
  vertex 149: (1, 1, 1, 3, 3, 4, 7)


\end{verbatim}

\section{Julia Output for \(n=8\)}
\begin{verbatim}
    DIRECTED DICE GRAPH — EXACT LONG-CYCLE SEARCH
========================================================================

n = 8
Required edge score n^2/2 + 1 = 33
Required cycle length >= n+1 = 9

========================================================================
1. VERTICES
========================================================================
Total number of dice/vertices: 6435
Expected binomial(2n-1,n):     6435
Vertex-count verification:      PASS

========================================================================
2. EDGES
========================================================================
Total number of directed edges: 956307
Required double-sum value:       33
Exact neighbor generation uses:
  sum_s b_s = n
  sum_s w_s b_s = n^2/2 + 1
  b_s >= 0
where w_s = #{i : x_i > s}.

Generated-edge count agrees with graph edge count: PASS
Every directed edge satisfies the original double-sum condition: PASS

========================================================================
3. STRONGLY CONNECTED COMPONENTS
========================================================================
Total number of SCCs found by Tarjan's algorithm: 2115

SCC sizes:
SCC 1: vertices = 4321, internal edges = 674108
2114 trivial SCCs (one vertex, zero internal edges each, edit to save room).

========================================================================
4. SCC PARTITION VERIFICATION
========================================================================
SCCs partition the full vertex set: PASS
Every directed cycle is therefore contained entirely within one SCC.
Number of SCCs having at least n+1 vertices: 1

========================================================================
5. SEARCH FOR A DIRECTED CYCLE OF LENGTH AT LEAST n+1
========================================================================
Cycle found: PASS
SCC containing cycle: 1
Size of containing SCC: 4321
Cycle length: 29

Cycle length >= n+1: PASS
Cycle has no repeated vertices: PASS
Every consecutive pair is a directed graph edge: PASS
Every cycle edge has double-sum value 33: PASS
All cycle vertices lie in one SCC: PASS

Dice/tuples:
  vertex 496: (1, 1, 1, 2, 4, 6, 6, 6)
  vertex 1257: (1, 1, 3, 3, 3, 3, 3, 5)
  vertex 3482: (2, 2, 2, 2, 2, 4, 4, 4)
  vertex 560: (1, 1, 1, 3, 3, 3, 7, 8)
  vertex 3456: (2, 2, 2, 2, 2, 2, 6, 7)
  vertex 507: (1, 1, 1, 2, 5, 5, 5, 6)
  vertex 2200: (1, 2, 3, 3, 3, 4, 4, 4)
  vertex 3476: (2, 2, 2, 2, 2, 3, 6, 6)
  vertex 508: (1, 1, 1, 2, 5, 5, 5, 7)
  vertex 2235: (1, 2, 3, 3, 4, 4, 4, 4)
  vertex 1816: (1, 2, 2, 2, 3, 3, 6, 6)
  vertex 506: (1, 1, 1, 2, 5, 5, 5, 5)
  vertex 2305: (1, 2, 3, 4, 4, 4, 4, 4)
  vertex 2187: (1, 2, 3, 3, 3, 3, 4, 6)
  vertex 3497: (2, 2, 2, 2, 2, 5, 5, 5)
  vertex 611: (1, 1, 1, 3, 4, 4, 8, 8)
  vertex 3477: (2, 2, 2, 2, 2, 3, 6, 7)
  vertex 509: (1, 1, 1, 2, 5, 5, 5, 8)
  vertex 2431: (1, 2, 4, 4, 4, 4, 4, 4)
  vertex 590: (1, 1, 1, 3, 3, 6, 7, 7)
  vertex 3498: (2, 2, 2, 2, 2, 5, 5, 6)
  vertex 677: (1, 1, 1, 4, 4, 4, 6, 7)
  vertex 1267: (1, 1, 3, 3, 3, 3, 5, 6)
  vertex 3490: (2, 2, 2, 2, 2, 4, 5, 8)
  vertex 593: (1, 1, 1, 3, 3, 7, 7, 7)
  vertex 3457: (2, 2, 2, 2, 2, 2, 6, 8)
  vertex 633: (1, 1, 1, 3, 5, 5, 5, 6)
  vertex 2662: (1, 3, 3, 3, 3, 4, 4, 4)
  vertex 3647: (2, 2, 2, 3, 3, 3, 3, 7)
  vertex 496: (1, 1, 1, 2, 4, 6, 6, 6)


\end{verbatim}

\section{Julia Output for \(n=9\)}
\begin{verbatim}
    ========================================================================

n = 9
Required edge score (n^2+1)/2  = 41
Required cycle length >= n+1 = 10

========================================================================
1. VERTICES
========================================================================
Total number of dice/vertices: 24310
Expected binomial(2n-1,n):     24310
Vertex-count verification:      PASS

========================================================================
2. EDGES
========================================================================
Total number of directed edges: 12528131
Required double-sum value:       41
Exact neighbor generation uses:
  sum_s b_s = n
  sum_s w_s b_s = (n^2+1)/2 
  b_s >= 0
where w_s = #{i : x_i > s}.

Generated-edge count agrees with graph edge count: PASS
Every directed edge satisfies the original double-sum condition: PASS

========================================================================
3. STRONGLY CONNECTED COMPONENTS
========================================================================
Total number of SCCs found by Tarjan's algorithm: 3860

SCC sizes:
SCC 1: vertices = 20451, internal edges = 10933011
SCC 2 through SCC 3860 are all singleton thus trivial (edited to save room)

========================================================================
4. SCC PARTITION VERIFICATION
========================================================================
SCCs partition the full vertex set: PASS
Every directed cycle is therefore contained entirely within one SCC.
Number of SCCs having at least n+1 vertices: 1

========================================================================
5. SEARCH FOR A DIRECTED CYCLE OF LENGTH AT LEAST n+1
========================================================================
Cycle found: PASS
SCC containing cycle: 1
Size of containing SCC: 20451
Cycle length: 10

Cycle length >= n+1: PASS
Cycle has no repeated vertices: PASS
Every consecutive pair is a directed graph edge: PASS
Every cycle edge has double-sum value 41: PASS
All cycle vertices lie in one SCC: PASS

Dice/tuples:
  vertex 830: (1, 1, 1, 1, 3, 3, 3, 3, 7)
  vertex 12875: (2, 2, 2, 2, 2, 2, 2, 2, 6)
  vertex 966: (1, 1, 1, 1, 3, 5, 5, 5, 5)
  vertex 3208: (1, 1, 2, 2, 2, 4, 4, 4, 4)
  vertex 1623: (1, 1, 1, 2, 3, 3, 3, 3, 8)
  vertex 12876: (2, 2, 2, 2, 2, 2, 2, 2, 7)
  vertex 890: (1, 1, 1, 1, 3, 3, 6, 6, 6)
  vertex 12887: (2, 2, 2, 2, 2, 2, 2, 4, 5)
  vertex 864: (1, 1, 1, 1, 3, 3, 4, 5, 9)
  vertex 12877: (2, 2, 2, 2, 2, 2, 2, 2, 8)
  vertex 830: (1, 1, 1, 1, 3, 3, 3, 3, 7)
\end{verbatim}

\section{Julia Output for \(n=10\)}
\begin{verbatim}
 ========================================================================

n = 10
Required edge score n^2/2 + 1 = 51
Required cycle length >= n+1 = 11

========================================================================
1. VERTICES
========================================================================
Total number of dice/vertices: 92378
Expected binomial(2n-1,n):     92378
Vertex-count verification:      PASS

========================================================================
2. EDGES
========================================================================
Total number of directed edges: 157244770
Required double-sum value:       51
Exact neighbor generation uses:
  sum_s b_s = n
  sum_s w_s b_s = n^2/2 + 1
  b_s >= 0
where w_s = #{i : x_i > s}.

Generated-edge count agrees with graph edge count: PASS
Every directed edge satisfies the original double-sum condition: PASS

========================================================================
3. STRONGLY CONNECTED COMPONENTS
========================================================================
Total number of SCCs found by Tarjan's algorithm: 10588

SCC sizes:
SCC 1: vertices = 81791, internal edges = 147620052
SCC 2 through 10588 were all trivial (edit to save room)

4. SCC PARTITION VERIFICATION
========================================================================
SCCs partition the full vertex set: PASS
Every directed cycle is therefore contained entirely within one SCC.
Number of SCCs having at least n+1 vertices: 1

========================================================================
5. SEARCH FOR A DIRECTED CYCLE OF LENGTH AT LEAST n+1
========================================================================
Cycle found: PASS
SCC containing cycle: 1
Size of containing SCC: 81791
Cycle length: 110

Cycle length >= n+1: PASS
Cycle has no repeated vertices: PASS
Every consecutive pair is a directed graph edge: PASS
Every cycle edge has double-sum value 51: PASS
All cycle vertices lie in one SCC: PASS


Dice/tuples:
  vertex 2828: (1, 1, 1, 1, 2, 4, 4, 4, 4, 4)
  vertex 27314: (1, 2, 2, 3, 3, 3, 3, 3, 3, 3)
  vertex 24410: (1, 2, 2, 2, 2, 2, 2, 4, 7, 7)
  vertex 3375: (1, 1, 1, 1, 3, 3, 3, 6, 6, 6)
  vertex 24420: (1, 2, 2, 2, 2, 2, 2, 5, 5, 5)
  vertex 2829: (1, 1, 1, 1, 2, 4, 4, 4, 4, 5)
  vertex 30746: (1, 2, 3, 3, 3, 3, 3, 3, 3, 3)
  vertex 5446: (1, 1, 1, 2, 2, 2, 6, 7, 7, 7)
  vertex 2913: (1, 1, 1, 1, 2, 4, 5, 5, 5, 6)
  vertex 8017: (1, 1, 1, 3, 3, 3, 3, 3, 4, 4)
  vertex 48687: (2, 2, 2, 2, 2, 2, 2, 3, 6, 6)
  vertex 3039: (1, 1, 1, 1, 2, 5, 5, 5, 5, 6)
  vertex 24596: (1, 2, 2, 2, 2, 2, 4, 4, 4, 4)
  vertex 8011: (1, 1, 1, 3, 3, 3, 3, 3, 3, 5)
  vertex 48705: (2, 2, 2, 2, 2, 2, 2, 4, 4, 7)
  vertex 3376: (1, 1, 1, 1, 3, 3, 3, 6, 6, 7)
  vertex 48709: (2, 2, 2, 2, 2, 2, 2, 4, 5, 5)
  vertex 3345: (1, 1, 1, 1, 3, 3, 3, 4, 7, 8)
  vertex 48652: (2, 2, 2, 2, 2, 2, 2, 2, 6, 7)
  vertex 3649: (1, 1, 1, 1, 3, 4, 4, 5, 5, 6)
  vertex 27315: (1, 2, 2, 3, 3, 3, 3, 3, 3, 4)
  vertex 24414: (1, 2, 2, 2, 2, 2, 2, 4, 8, 8)
  vertex 3380: (1, 1, 1, 1, 3, 3, 3, 6, 7, 7)
  vertex 48730: (2, 2, 2, 2, 2, 2, 2, 5, 5, 5)
  vertex 5831: (1, 1, 1, 2, 2, 4, 4, 4, 4, 4)
  vertex 8012: (1, 1, 1, 3, 3, 3, 3, 3, 3, 6)
  vertex 48706: (2, 2, 2, 2, 2, 2, 2, 4, 4, 8)
  vertex 3346: (1, 1, 1, 1, 3, 3, 3, 4, 7, 9)
  vertex 48653: (2, 2, 2, 2, 2, 2, 2, 2, 6, 8)
  vertex 3040: (1, 1, 1, 1, 2, 5, 5, 5, 5, 7)
  vertex 24926: (1, 2, 2, 2, 2, 3, 4, 4, 4, 4)
  vertex 8013: (1, 1, 1, 3, 3, 3, 3, 3, 3, 7)
  vertex 24427: (1, 2, 2, 2, 2, 2, 2, 5, 6, 7)
  vertex 3648: (1, 1, 1, 1, 3, 4, 4, 5, 5, 5)
  vertex 51625: (2, 2, 2, 3, 3, 3, 3, 3, 3, 4)
  vertex 5447: (1, 1, 1, 2, 2, 2, 6, 7, 7, 8)
  vertex 3705: (1, 1, 1, 1, 3, 4, 5, 5, 5, 6)
  vertex 14452: (1, 1, 2, 3, 3, 3, 3, 3, 4, 4)
  vertex 48688: (2, 2, 2, 2, 2, 2, 2, 3, 6, 7)
  vertex 3041: (1, 1, 1, 1, 2, 5, 5, 5, 5, 8)
  vertex 25136: (1, 2, 2, 2, 2, 4, 4, 4, 4, 4)
  vertex 3347: (1, 1, 1, 1, 3, 3, 3, 4, 7, 10)
  vertex 48654: (2, 2, 2, 2, 2, 2, 2, 2, 6, 9)
  vertex 3706: (1, 1, 1, 1, 3, 4, 5, 5, 5, 7)
  vertex 27322: (1, 2, 2, 3, 3, 3, 3, 3, 4, 4)
  vertex 24415: (1, 2, 2, 2, 2, 2, 2, 4, 8, 9)
  vertex 3390: (1, 1, 1, 1, 3, 3, 3, 7, 7, 7)
  vertex 24426: (1, 2, 2, 2, 2, 2, 2, 5, 6, 6)
  vertex 4110: (1, 1, 1, 1, 4, 4, 4, 5, 5, 5)
  vertex 8018: (1, 1, 1, 3, 3, 3, 3, 3, 4, 5)
  vertex 48689: (2, 2, 2, 2, 2, 2, 2, 3, 6, 8)
  vertex 3042: (1, 1, 1, 1, 2, 5, 5, 5, 5, 9)
  vertex 25718: (1, 2, 2, 2, 3, 3, 4, 4, 4, 4)
  vertex 8014: (1, 1, 1, 3, 3, 3, 3, 3, 3, 8)
  vertex 24428: (1, 2, 2, 2, 2, 2, 2, 5, 6, 8)
  vertex 4090: (1, 1, 1, 1, 4, 4, 4, 4, 5, 6)
  vertex 14453: (1, 1, 2, 3, 3, 3, 3, 3, 4, 5)
  vertex 24417: (1, 2, 2, 2, 2, 2, 2, 4, 9, 9)
  vertex 3391: (1, 1, 1, 1, 3, 3, 3, 7, 7, 8)
  vertex 24441: (1, 2, 2, 2, 2, 2, 2, 6, 6, 6)
  vertex 3043: (1, 1, 1, 1, 2, 5, 5, 5, 5, 10)
  vertex 25928: (1, 2, 2, 2, 3, 4, 4, 4, 4, 4)
  vertex 3349: (1, 1, 1, 1, 3, 3, 3, 4, 8, 9)
  vertex 48657: (2, 2, 2, 2, 2, 2, 2, 2, 7, 8)
  vertex 2969: (1, 1, 1, 1, 2, 4, 6, 6, 6, 7)
  vertex 8024: (1, 1, 1, 3, 3, 3, 3, 3, 5, 5)
  vertex 24418: (1, 2, 2, 2, 2, 2, 2, 4, 9, 10)
  vertex 3394: (1, 1, 1, 1, 3, 3, 3, 7, 8, 8)
  vertex 48655: (2, 2, 2, 2, 2, 2, 2, 2, 6, 10)
  vertex 3707: (1, 1, 1, 1, 3, 4, 5, 5, 5, 8)
  vertex 30747: (1, 2, 3, 3, 3, 3, 3, 3, 3, 4)
  vertex 5448: (1, 1, 1, 2, 2, 2, 6, 7, 7, 9)
  vertex 3831: (1, 1, 1, 1, 3, 5, 5, 5, 5, 6)
  vertex 5832: (1, 1, 1, 2, 2, 4, 4, 4, 4, 5)
  vertex 8015: (1, 1, 1, 3, 3, 3, 3, 3, 3, 9)
  vertex 24429: (1, 2, 2, 2, 2, 2, 2, 5, 6, 9)
  vertex 4091: (1, 1, 1, 1, 4, 4, 4, 4, 5, 7)
  vertex 8019: (1, 1, 1, 3, 3, 3, 3, 3, 4, 6)
  vertex 48690: (2, 2, 2, 2, 2, 2, 2, 3, 6, 9)
  vertex 3044: (1, 1, 1, 1, 2, 5, 5, 5, 6, 6)
  vertex 26390: (1, 2, 2, 2, 4, 4, 4, 4, 4, 4)
  vertex 6353: (1, 1, 1, 2, 3, 3, 3, 4, 8, 10)
  vertex 48658: (2, 2, 2, 2, 2, 2, 2, 2, 7, 9)
  vertex 2970: (1, 1, 1, 1, 2, 4, 6, 6, 6, 8)
  vertex 14446: (1, 1, 2, 3, 3, 3, 3, 3, 3, 5)
  vertex 24419: (1, 2, 2, 2, 2, 2, 2, 4, 10, 10)
  vertex 3392: (1, 1, 1, 1, 3, 3, 3, 7, 7, 9)
  vertex 48710: (2, 2, 2, 2, 2, 2, 2, 4, 5, 6)
  vertex 3352: (1, 1, 1, 1, 3, 3, 3, 4, 9, 10)
  vertex 48661: (2, 2, 2, 2, 2, 2, 2, 2, 8, 9)
  vertex 2974: (1, 1, 1, 1, 2, 4, 6, 6, 7, 8)
  vertex 17878: (1, 1, 3, 3, 3, 3, 3, 3, 3, 5)
  vertex 24597: (1, 2, 2, 2, 2, 2, 4, 4, 4, 5)
  vertex 8016: (1, 1, 1, 3, 3, 3, 3, 3, 3, 10)
  vertex 24430: (1, 2, 2, 2, 2, 2, 2, 5, 6, 10)
  vertex 4092: (1, 1, 1, 1, 4, 4, 4, 4, 5, 8)
  vertex 8020: (1, 1, 1, 3, 3, 3, 3, 3, 4, 7)
  vertex 24442: (1, 2, 2, 2, 2, 2, 2, 6, 6, 7)
  vertex 3499: (1, 1, 1, 1, 3, 3, 5, 5, 5, 10)
  vertex 48707: (2, 2, 2, 2, 2, 2, 2, 4, 4, 9)
  vertex 3393: (1, 1, 1, 1, 3, 3, 3, 7, 7, 10)
  vertex 48715: (2, 2, 2, 2, 2, 2, 2, 4, 6, 6)
  vertex 3708: (1, 1, 1, 1, 3, 4, 5, 5, 5, 9)
  vertex 51632: (2, 2, 2, 3, 3, 3, 3, 3, 4, 4)
  vertex 5449: (1, 1, 1, 2, 2, 2, 6, 7, 7, 10)
  vertex 4111: (1, 1, 1, 1, 4, 4, 4, 5, 5, 6)
  vertex 8021: (1, 1, 1, 3, 3, 3, 3, 3, 4, 8)
  vertex 24432: (1, 2, 2, 2, 2, 2, 2, 5, 7, 8)
  vertex 3634: (1, 1, 1, 1, 3, 4, 4, 4, 6, 7)
  vertex 49910: (2, 2, 2, 2, 3, 3, 3, 3, 3, 5)
  vertex 2828: (1, 1, 1, 1, 2, 4, 4, 4, 4, 4)

\end{verbatim}

\end{document}